\documentclass[12pt,twoside]{amsart}
\usepackage{mathrsfs}
\usepackage{mathtools}
\usepackage{amssymb}
\usepackage{verbatim}
\usepackage{amsmath}
\usepackage{comment}
\usepackage{amsthm,thmtools,xcolor}
\usepackage[colorlinks,linkcolor=blue,citecolor=blue, pdfstartview=FitH]
{hyperref}
\usepackage{bm}
\usepackage{a4wide}
\usepackage[latin1]{inputenc}
\usepackage[T1]{fontenc}
\usepackage{times}
\usepackage{hyperref}
\usepackage{amssymb,latexsym}
\usepackage{enumerate}

\usepackage{etoolbox}
\makeatletter
\patchcmd\maketitle
  {\uppercasenonmath\shorttitle}
  {}
  {}{}
\patchcmd\maketitle
  {\@nx\MakeUppercase{\the\toks@}}
  {\the\toks@}
  {}
  {}{}
\patchcmd\@settitle
  {\uppercasenonmath\@title}
  {}
  {}{}
\patchcmd\@setauthors
  {\MakeUppercase{\authors}}
  {\authors}
  {}{}

\def\Re{{\rm Re}}
\def\Im{{\rm Im}}

\makeatletter
\newcommand{\sumprime}{\if@display\sideset{}{'}\sum%
            \else\sum'\fi}
\makeatother

\begin{document}

\numberwithin{equation}{section}

\newtheorem{theorem}{Theorem}[section]
\newtheorem{proposition}[theorem]{Proposition}
\newtheorem{conjecture}[theorem]{Conjecture}
\def\theconjecture{\unskip}
\newtheorem{corollary}[theorem]{Corollary}
\newtheorem{lemma}[theorem]{Lemma}
\newtheorem{observation}[theorem]{Observation}
\newtheorem{definition}{Definition}
\numberwithin{definition}{section} 
\newtheorem{remark}{Remark}
\def\theremark{\unskip}
\newtheorem{kl}{Key Lemma}
\def\thekl{\unskip}
\newtheorem{question}{Question}
\def\thequestion{\unskip}
\newtheorem{example}{Example}
\def\theexample{\unskip}
\newtheorem{problem}{Problem}

\address{Franz Luef, Johannes Testorf, Xu Wang: Department of Mathematical Sciences, Norwegian University of Science and Technology, Trondheim, Norway}

\email{franz.luef@ntnu.no, johannes.testorf@ntnu.no, xu.wang@ntnu.no}

\title[Transcendentality condition for Gaussian Gabor frames]{On the transcendentality condition for Gaussian Gabor frames and Hermite super/multiwindow frames}

 \author{Franz Luef, Johannes Testorf, Xu Wang}
\date{\today}

\begin{abstract} We give a criterion for higher-dimensional Gaussian Gabor frames, which is a reformulation of one of the main results in \cite[Thm 1.1]{LW} in more explicit terms. We use this formulation in order to extend the result of \cite{RUZ} to lattices given by irrational rotations. We also show that this density criterion for Gaussian Gabor frames is generic in a certain sense. In addition, we also extend the methods of \cite{LW} to the pseudoeffective threshold which gives a condition for uniqueness in the Bargmann-Fock space. We also use this viewpoint to study super and multi-window Gabor frames with Hermitian windows. In particular, we find a density criterion for transcendental lattices.
\end{abstract}

\maketitle

\section{Introduction}
We continue the research on Gaussian Gabor frames using methods from K\"ahler geometry, initiated in \cite{LW}, and extend it the study of Hermite multi-window/super Gabor frames by results on jet interpolation in the Bargmann-Fock spaces.

Let us recall the setting of \cite{LW}. We take a Gaussian $g_\Omega(t):=\overline{e^{\pi it^T\Omega t}}$, where $\Omega\in \mathfrak{h}$ is an element of the Siegel upper half space
$$
\mathfrak{h}:=\{\Omega\in \mathfrak{gl}(n,\mathbb C): \Omega = \Omega^T,\, \mathrm{Im}\,\Omega \text{ is positive definite}\},
$$
a lattice $\Gamma \subset \mathbb R^{2n}$, and associate to these the Gaussian Gabor system $\{\pi_\lambda g_\Omega\}_{\lambda\in\Lambda}$ as
$$
\pi_\lambda g_\Omega(t):=e^{2\pi i \xi^Tt}g_\Omega(t-x),\quad \lambda = (\xi,x)\in \Lambda, \quad\xi^Tt:=\sum\limits_{j=1}^n\xi_jt_j.
$$
Then $\{\pi_\lambda g_\Omega\}_{\lambda\in\Lambda}$ is called a Gaussian Gabor frame if there exist constants $A,B>0$ such that
$$
A\|f\|^2\le\sum\limits_{\lambda\in \Lambda}|(f,\pi_\lambda g_\Omega)|^2\le B\|f\|^2 \quad \text{for } f\in L^2(\mathbb R^{n}),
$$
where $(\cdot,\cdot)$ is the $L^2$ inner product.

Multivariate Gaussian Gabor frames have recently been investigated by the first and the last author of this note \cite[Thm. 1.1]{LW}, where a density condition for a certain class of Gaussian Gabor frames is formulated for transcendental lattices $\Lambda$ in $\mathbb{R}^{2n}$. Recall that a lattice $\Gamma$ in $\mathbb C^n$ is said to be transcendental if the complex  torus $X:=\mathbb C^n/\Gamma$ is Campana simple, i.e. the only positive dimensional analytic subvariety of $X$ is $X$ itself.

We will generalize the results in the case of such a lattice when $\Omega=iI$ (By Proposition 1.4 in \cite{LW} we do not any lose generality in the Gaussian case.) to Hermite functions. We define these as follows:
\begin{definition}
    For $m\in \mathbb N_0$, we define the Hermite function of level $m$ by 
    $$
    h_m(t):=c_m e^{\pi |t|^2}\frac{d^m}{dt^m}\left(e^{-2\pi |t|^2} \right), \quad t\in \mathbb R
    $$
    where $c_m$ is a constant ensuring that the $L^2$ norm of the Hermite functions is 1.  
    \\~~\\
    For $\alpha\in \mathbb N_0^n$, we define the Hermite functions of level $\alpha$ by
    $$
    h_\alpha(t):=\prod\limits_{i=1}^n h_{\alpha_i}(t_i),\quad t\in \mathbb R^n.
    $$
\end{definition}
We also define the notions of super and multiwindow Gabor frames.
\begin{definition}
    Let $\Lambda \subset\mathbb R^{2n}$ be a lattice. The system 
    $(\mathbf{h}_s:=(h_\alpha)_{|\alpha|\le s},\Lambda)$ induces an s-super-Gabor frame, if there exist constants $A,B>0$ such that for any $f\in L^2(\mathbb R^n,\mathbb C^{\mathfrak{s} })$ where $\mathfrak s$ is the number of multiindices with $|\alpha|\le s$,
    $$
    A\|f\|^2\le \sum\limits_{\lambda\in\Lambda}\,\bigg|\sum\limits_{|\alpha|\le s}(f_\alpha,\pi_\lambda h_\alpha)\, \bigg|^2 \le B\|f\|^2.
    $$
    The system $(\mathbf{h}_s:=(h_\alpha)_{|\alpha|\le s},\Lambda)$, induces a multiwindow Gabor frame if there exist constants $A,B>0$ such that for every $f\in L^2(\mathbb R^n)$
    $$
    A\|f\|^2\le \sum\limits_{\lambda\in \Lambda}\sum\limits_{|\alpha|\le s}|(f,\pi_\lambda h_\alpha))|^2\le B\|f\|^2.
    $$
\end{definition}

Then our main result is the following generic density criterion for Gabor frames.
\begin{theorem}\label{th:mainthm} We have that: 
\begin{enumerate}
\item[(1)] $(\mathbf h_s, \Lambda)$ is a multiwindow-Gabor frame for $L^2(\mathbb R^n)$ if $|\Lambda| <\frac {(s+1)^n}{n!}$ and $\ker {\rm int}_{\Lambda_{\mathbb C}} =0$.

\item[(2)] $(\mathbf h_s, \Lambda)$ is a super Gabor frame for $L^2(\mathbb R^n, \mathbb C^{\mathfrak s})$ if $|\Lambda| <\frac{n!}{(n+s)^n}$ and $\ker {\rm int}_{\Lambda^\circ_{\mathbb C}} =0$.
\end{enumerate}
In particular, when $s=0$, both (1) and (2) will give a criterion for $(e^{-\pi |t|^2},\Lambda)$ to be a Gabor frame.
\end{theorem}
The map $\mathrm{int}_{\Lambda_\mathbb C}$ is defined in (\ref{eq:sha}) and the vanishing of its kernel is equivalent to the transcendentality condition for $\Lambda$ and its symplectic dual as will be proved later on. 

We remark that this result not only generalizes the results of \cite{LW} but also the results of \cite{Abreu10}, as when $n=1$ this will reduce to precisely the result given there.

One of the main objects used in \cite{LW} to obtain the density criterion is the Seshadri constant. Here, we will investigate the related notion of the pseudoeffective threshold as well.  The proof of the above theorem relies on the characterization of these numbers give of sampling and uniqueness in the Bargmann-Fock space. In particular, we want to give a sort of time-frequency interpretation of these numbers. In order to add some detail to this interpretation we also generalize asymptotic formulae of Demailly for the Seshadri constant and pseudoeffective threshold in terms of interpolation and uniqueness to the Bargmann-Fock case.

As a consequence of our work in this paper, the duality theorem of Gabor frames can also be used to relate the pseudoeffective threshold and Seshadri constant on tori given by dual lattices, however, the relation this gives is very weak (in the transcendental case this is even weaker than $|\Lambda|\cdot|\Lambda^o|=1$), whether there is some method which can give a more intimate relation between these (in particular a sharper bound of the Seshadri constant in terms of the pseudoeffective threshold of the dual torus or vice versa) remains open.

We will also illustrate how the notion of the Seshadri constant and pseudoeffective threshold are connected to density for transcendental lattices and will give a characterization of transcendental lattices purely in terms of the lattice points and we will also demonstrate that the transcendentality criterion is a generic property.

By our reformulation of \cite[Thm. 1.1]{LW} in purely combinatorial terms for the lattice we are able to cover lattices outside the scope of the density conditions for bivariate Gaussian Gabor frames for product lattices established in \cite{RUZ}. More precisely, we are able to give a density condition for product lattices given by an irrational rotation (see Corollary \ref{cor:RUZ}).

\section{Pseudoeffective Thresholds and Gaussian Gabor Frames}
Recall that for a K\"ahler manifold $(X,\omega)$ the Seshadri constant at a point $x\in X$ is defined as 
    \begin{align*}
    \epsilon_x(X,\omega):=\sup\{\gamma \ge 0 : \exists\, \phi \in \mathrm{PSH}(X,\omega) \text{ such that $\phi = \gamma \log|z|^2+\mathcal{O}(1)$ near $x$}\},
\end{align*}
where z denotes a holomorphic coordinate chart around $x$ with $z(x)=0$ and $\mathrm{PSH}(X,\omega)$ denotes the the space of $\omega$-plurisubharmonic (PSH) functions, i.e. the space of functions $\phi$ on $X$ such that for every local potential $\psi$ of $\omega$ (i.e. $\omega= dd^c \psi$, $dd^c:= \tfrac{i}{2\pi} \partial\overline{\partial}$) $\phi+\psi$ is subharmonic on each holomorphic disk embedded in $X$. We similarly define the pseudoeffective threshold at $x$ by 
\begin{align*}
    \lambda_x(X,\omega):=\sup\{\gamma\ge 0:\exists \phi \in \mathrm{PSH}(X,\omega), \, \nu_x(\phi)\ge \gamma \},
\end{align*}
where $\nu_x(\phi)$ is the Lelong number at $x$ defined by
\begin{align*}
    \nu_x(\phi):=\liminf\limits_{z\to 0}\frac{\phi(z)}{\log|z|^2}.
\end{align*}

We shall first give a condition on lattices $\Lambda\subseteq\mathbb{R}^{2n}$ that give Gaussian Gabor frames $(g_\Omega,\Lambda)$ via the pseudoeffective threshold.
By Proposition 1.4 in \cite{LW}, it suffices to look at the $\Omega=i I$ case, i.e.
$$
g_\Omega(t)=e^{-\pi|t|^2}.
$$
In this case, by the Bargmann transform, we know that $(e^{-\pi|t|^2}, \Lambda)$ is a Gabor frame for $L^2(\mathbb R^n)$ if and only if 
$$
\Lambda_{\mathbb C}:=\{x+i\xi \in \mathbb C^n: (\xi, x)\in \Lambda\}
$$
defines a frame for the classical Bargmann-Fock space. Hence, by \cite[Prop. 2.1]{GL} we have the following fact:

\begin{proposition}\label{pr:GL} $(e^{-\pi|t|^2}, \Lambda)$ is a Gabor frame for $L^2(\mathbb R^n)$ if and only if $\Lambda_{\mathbb C}$ is a set of uniqueness for $\mathcal F^\infty$, i.e. if for every  holomorphic function $F$ on $\mathbb C^n$ we have  
$$
\sup_{z\in \mathbb C^n} |F(z)|^2 e^{-\pi|z|^2}\le1 \ \ \text{and} \ \ F|_{\Lambda_\mathbb C} =0,
$$
then $F\equiv 0$ on $\mathbb C^n$.
\end{proposition}

On the other hand, one may use K\"ahler geometry methods to prove the special case of our main theorem.

\begin{proposition}\label{pr:PT}
    If $\lambda_0(\mathbb C^n/\Lambda_{\mathbb C},\omega)<1$, with $\omega$ being the lifting of the standard K\"ahler metric $dd^c\pi |z|^2$ of $\mathbb C^n$ to the torus, then the lattice $\Lambda_{\mathbb C}$ is a set of uniqueness for $\mathcal{F}^\infty$.
\end{proposition}

\begin{proof}
If $\Lambda_{\mathbb C}$ is not a set of uniqueness for $\mathcal F^\infty$,  then the following $\omega$-PSH function
\begin{align*}
G(z):=\sup^*\left\lbrace \log(|F(z)|^2 e^{-\pi|z|^2}): F\in \mathcal O(\mathbb C^n), \sup_{z\in \mathbb C^n} |F(z)|^2 e^{-\pi|z|^2}=1, \ F|_{\Lambda_\mathbb C} =0\right\rbrace
\end{align*}
is not identically equal to $-\infty$, where $\mathcal O(\mathbb C^n)$ denotes the space of holomorphic functions on $\mathbb C^n$, $\sup^*$ denotes the upper semicontinuous regularization of the supremum. Note that $G$ is $\Lambda_{\mathbb C}$-invariant. 
\\~\\
To see this, consider a point $z \in \mathbb{C}^n$, a lattice point $\lambda \in \Lambda_\mathbb{C}$, and a holomorphic function $F$ which is a suitable candidate for $G$. Then we have
$$
|F(z+\lambda)|^2e^{-\pi|z+\lambda|^2}= |F(z+\lambda)e^{-\pi z\overline{\lambda}-\frac{\pi|\lambda|^2}{2}}|^2e^{-\pi |z|^2}.
$$
Let us define
$$
\tilde{F}(z):= F(z+\lambda)e^{-\pi z\overline{\lambda}-\frac{\pi|\lambda|^2}{2}},
$$
we have that $\tilde{F}$ is also a candidate for $G$. This means that $G(z)$ is no less than $G(z+\lambda)$, which implies $\Lambda_\mathbb C$-invariance. Thus, $G$ induces a function on the torus. Since by definition, $G$ will have Lelong number of at least one on the lattice point, thus the statement is proved.
\end{proof}

\section{Jet Interpolation and Uniqueness in the Bargmann-Fock Space}
Here we will characterize the pseudoeffective threshold in terms of uniqueness sets for jets of holomorphic functions. We will follow what was laid out in \cite{Dem08} for line bundles, and show a similar result for the Bargmann-Fock space with respect to a lattice. 

We will mimic tensor powers of complex line bundles via the spaces 
\begin{align*}
    \mathcal{F}^\infty_k:=\{F\in \mathcal O(\mathbb C^n) : \|F\|_{\mathcal{F}^\infty_k} <\infty\},\quad k\in \mathbb R_{>0},
\end{align*}
 where 
\begin{align*}
    \|F\|_{\mathcal{F}^\infty_k}:=\sup\limits_{z\in\mathbb C^n}|F(z)|^2e^{-\pi k |z|^2}.
\end{align*}

We say that the lattice $\Lambda_\mathbb C$ is a set of uniqueness for $s$-jets of $\mathcal{F}^\infty_k$ if for every $F\in\mathcal{F}^\infty_k$ with
\begin{align*}
    \left.\frac{\partial^\alpha}{\partial z^\alpha}F(z)\right|_{\Lambda_\mathbb C} =0, \quad \alpha\in \mathbb N_0^n, \, |\alpha|\le s,
\end{align*}
we have $F\equiv 0$ on $\mathbb C^n$. 
\\~~\\
We can thus introduce the uniqueness numbers associated to the lattice $\Lambda_\mathbb C$ as 
$$
\mu_k(\Lambda_\mathbb{C}):=\sup\{s+1\in \mathbb N_0 : \Lambda_\mathbb C \text{ is not a set of uniqueness for s-jets for $\mathcal{F}^\infty_k$}\}.
$$
Since there are no $(-1)$-jets, we always have $\mu_k\ge 0$. By generalizing the proof of Proposition \ref{pr:PT} to $G$ defined by
\begin{align*}
G(z):=\sup^*\left\lbrace \log(|F(z)|^2 e^{-\pi k|z|^2}):F\in \mathcal O(\mathbb C^n), \|F\|_{\mathcal{F}^\infty_k} =1, \ \left.\frac{\partial^\alpha F}{\partial z^\alpha}\right|_{\Lambda_\mathbb C} =0, |\alpha|\le \mu_k \right\rbrace,
\end{align*}
we may see that by linearity of the pseudoeffective threshold with respect to the metric, we have 
\begin{align}\label{eq:PSTI}
\lambda_0(\mathbb C^n/\Lambda_\mathbb C,\omega) \ge \frac{\mu_k(\Lambda_\mathbb{C})}{k}, \quad k\in \mathbb R_{>0}.
\end{align}
\medskip

\noindent
\textbf{Remark.}  The space $\mathcal{F}^\infty_k$ is the $L^\infty$ Bargmann-Fock space associated to the matrix $ikI\in \mathfrak h$. By considerations which can be found in \cite{LW}, this can also be interpreted as multiplying the lattice $\Lambda_\mathbb C$ by the factor $\sqrt{k}$ rather than a change of the weight. Thus, $\mu_k(\Lambda_\mathbb{C}) = 0$, is equivalent to $\sqrt{k}\Lambda_\mathbb C$ being a set of uniqueness for the Bargmann-Fock space. Furthermore, the asymptotic behavior in what follows may be interpreted in terms of rescaling of a given lattice.
\\~~\\
One can show that the inequality (\ref{eq:PSTI}) is asymptotically an equality. That is the following.

\begin{lemma}
    For any complex lattice $\Lambda_\mathbb C$ we have that 
    $$
    \lambda_0(\mathbb C^n/\Lambda_\mathbb C,\omega) = \lim\limits_{k\to \infty} \frac{\mu_k(\Lambda_\mathbb{C})}{k}.
    $$
\end{lemma}

\begin{proof}
    Our proof is very much in the spirit of the one Demailly gives in \cite{Dem08}. The main difference is that we lose compactness of our base manifold. 
		\\~~\\
		To prove the statement, we must find that for any $k$ sufficiently large, there exists a function $F\in\mathcal F^\infty_k$ whose derivatives of sufficiently large order vanish on $\Lambda_\mathbb C$. To this end, choose a number $\gamma\in [0,\lambda_0(\mathbb C^n/\Lambda_\mathbb C,\omega))$. By definition there exists a negative strictly $\omega$-PSH function (means $dd^c (\cdot) +\omega$ is strictly positive), say, $\psi$ on $\mathbb C^n/\Lambda_\mathbb C$ which has Lelong number of at least $\lambda$ at the lattice point (later we sill identify $\psi$ with its pull back to $\mathbb C^n$).
		\\~~\\
		Now choose $U \Subset U'\subseteq \mathbb C^n/\Lambda_\mathbb C$ two sufficently small neighborhoods of the lattice point, and a smooth function $\theta$ which is equal to one on the closure of $U$ and equal to zero outside $U'$. We denote by $\pi$ the projection map from $\mathbb C^n$ to the torus, and the leaves of $\pi^{-1}(U')$ by $(U'_\lambda)_{\lambda\in \Lambda_\mathbb C}.$  We define the sets $U_\lambda$ similarly. Fix a point $z_0\in U_0$ outside the singular locus of $\psi$, let us choose a sequence $(a_{k,\lambda})_{\lambda \in \Lambda_\mathbb C}$ such that
		$$
		\sum\limits_{\lambda \in \Lambda_\mathbb C} |a_{k,\lambda}|^2 e^{-\pi k |\lambda+z_0|^2-k\psi(\lambda+z_0)} = 1.
		$$
		By the Ohsawa-Takegoshi extension theorem, there exists a holomorphic function $f_\lambda$ on each $U_\lambda'$ such that $f_\lambda(z_0+\lambda) = a_{k,\lambda},$ and the estimate 
		$$
		\int\limits_{U_\lambda'}|f_\lambda|^2e^{-\pi k |z|^2-\psi(z)}\le C_1|a_{k,\lambda}|^2e^{-\pi k |\lambda+z_0|^2-k\psi(\lambda+z_0)},
		$$
		is fulfilled for some $C_1 = C_1(n,U').$ 
		\\~~\\
		In order to gain a global holomorphic function, we choose a cutoff function $\theta$ around the lattice point in the torus and solve the equation 
		$$
		\overline{\partial}g = \overline\partial f, \  \ \ f:= \left((\theta\circ \pi)\cdot \sum\limits_{\lambda\in \Lambda_\mathbb C} 1_{U_\lambda'}f_\lambda\right)
		$$
		using the Andreotti-Versentini-H\"ormander estimates. We may do this since all that is required from our manifold for this is weak pseudoconvexity which we have for $\mathbb C^n$. 
		We solve the equation with respect to the metric 
		$$
		\phi(z):=\pi k |z|^2+k\psi + \sum_{\lambda\in \Lambda_\mathbb C}\left( 2n\chi(|z-z_0-\lambda|) \log |z-z_0-\lambda|\right),
		$$
		where $\chi$ is a cuttoff function on $\mathbb R_+$ which is equal to $1$ near $0$ and vanishes for values greater than the injectivity radius of $z_0$ in the torus.
		\\~~\\
		In order to apply the Andreotti-Versentini-H\"ormander estimates, we require that the curvature of the weight function is larger than $\varepsilon \omega$ for some $\varepsilon >0$. This is fulfilled for $k$ sufficiently large since $\psi$ has been defined to be strictly $\omega$-PSH, and since the curvature of the metric lifts to the torus, which is compact, we may choose the lower bound for the $k$ which fulfill this condition uniformly with respect to $z_0$. Note that the precise choice of $\chi$ will influence how large we must choose $k$. Thus the $L^2$ estimates give
    $$
		\int\limits_{\mathbb C^n}|g|^2e^{-\phi}\le C_2\sum\limits_{\lambda\in \Lambda}\int\limits_{U_\lambda'}\left|\overline\partial (\theta\cdot f_\lambda)\right|^2 e^{-\phi}.
		$$
    Since $\phi$ has a singularity at $z_0+\lambda$ for every $\lambda \in \Lambda_\mathbb C$, we can see that $g(z_0+\lambda)=0$ and if the $U$ was chosen small enough we get for 
   $$F:=f-g,$$
		that
		\begin{align*}
			\int\limits_{\mathbb C^n}|F|^2e^{-\pi k |z|^2}\le C_3\sum\limits_{\lambda\in \Lambda} \int\limits_{U_\lambda'}|f_\lambda|^2e^{-\pi k |z|^2-k\psi(z)}\le C_1C_3<\infty.
		\end{align*}
		The fact that the integral on the LHS is finite implies that $F$ lies in the Bargmann-Fock space associated to $ikI$, which implies by (see e.g. \cite{Fol}) that $F$ lies in $\mathcal{F}^\infty_k$. $F$ also has vanishing order of at least $\lfloor k\gamma \rfloor$. Thus we have that
		$$
		\mu_k(\Lambda_\mathbb C)\ge\lfloor k\gamma \rfloor.
		$$
		Since $\gamma$ can be chosen arbitrarily close to the pseudoeffective threshold we are done.
\end{proof}

\noindent
We have a similar characterization of the Seshadri constant in terms of interpolation. The case of a line bundle over a compact manifold was shown in \cite{Dem-sin}. We say that the lattice $\Lambda_\mathbb C$ interpolates $s$-jets in $\mathcal{F}^2_k$ if for any sequence of complex numbers $(a_{\alpha,\lambda})_{\lambda \in \Lambda_\mathbb C, \alpha\in N_s}$ with 
$$
\sum\limits_{\lambda\in\Lambda_\mathbb C, \alpha\in N_s}|a_{\alpha,\lambda}|^2e^{-\pi k |\lambda|^2}=1,
$$
there exists an entire function $F$ with 
$$
\|F\|_{\mathcal{F}^2_k}:=\int_{\mathbb C^n} |F|^2e^{-k\pi |z|^2} <\infty,
$$
such that 
$$
e^{\pi k|z|^2}\frac{d^\alpha }{dz^\alpha}(F(z)e^{-\pi k|z|^2})|_{z=\lambda} = a_{\alpha,\lambda}, \quad \alpha\in N_s, \lambda\in \Lambda_\mathbb C.
$$
The index set $N_s$ is defined by $N_s:=\{\alpha\in \mathbb N^d_0 : |\alpha|\le s\}.$ 
\\~~\\
Similarly to the uniqueness numbers, we define the interpolation numbers associated to the lattice $\Lambda_\mathbb C$ by
$$
\sigma_k(\Lambda_\mathbb C):=\sup\{s\in\mathbb N_0 : \Lambda_\mathbb C \text{ interpolates s-jets in $\mathcal F^2_k$}\},
$$
if $\Lambda_\mathbb C$ does not interpolate in the space $\mathcal{F}^2_k$, we set $\sigma_k$ to -1.
\\~~\\
Analogously to before, if $\Lambda_\mathbb C$ is a set of s-jet interpolation for $\mathcal F^2_k$, then the envelope
$$
\sup^*\left\lbrace \log(|F(z)|^2 e^{-\pi k|z|^2}):F\in \mathcal{O}(\mathbb C^n), \|F\|_{\mathcal{F}^2_k} =1, \ \left.\frac{\partial^\alpha F}{\partial z^\alpha}\right|_{\Lambda_\mathbb C} =0, |\alpha|< \sigma_k(\Lambda_\mathbb C) \right\rbrace,
$$
has isolated logarithmic singularities with Lelong number $\sigma_k(\Lambda_\mathbb C)$ on the lattice since for any direction out of a given lattice point we may find a jet of level $\sigma_k(\Lambda_{\mathbb C})$ which is nonzero in that direction. The envelope also descends to the torus similarly to the envelopes we have previously studied. Thus we have
$$
\epsilon_0(\mathbb C^n/\Lambda_\mathbb C,\omega) \ge \frac{\sigma_k(\mathbb C^n/\Lambda_\mathbb C,\omega)}{k}.
$$

Like with the pseudoeffective threshold, we can show equality in the limit.
\begin{lemma}
    For any complex lattice $\Lambda_\mathbb C$ we have that 
    $$
   \epsilon_0(\mathbb C^n/\Lambda_\mathbb C,\omega) = \lim\limits_{k\to \infty} \frac{\sigma_k(\Lambda_\mathbb{C})}{k}.
    $$
\end{lemma}
We remark that the proof is essentially the same as \cite[Theorem 2.10]{LW}, the main difference being that since we do not assume our lattice to be transcendental we must write the criterion in terms of the Seshadri constant rather than the density of the lattice.
\begin{proof}
		If we have $k \epsilon_0(\mathbb C^n/\Lambda_\mathbb C,\omega)>n+s$, we may find an $k\omega$-PSH function $\psi$ on $\mathbb C^n/\Lambda_\mathbb C$ such that $\psi$ is of the form $\gamma\log|z|^2+\mathcal{O}(1)$ near the lattice point for any $n+s<\gamma<k \epsilon_0(\mathbb C^n/\Lambda_\mathbb C,\omega).$ This function can be lifted to a $\Lambda$ invariant function on $\mathbb C^n$.
		\\~~\\
		Now choose some sequence $(a_{k,\alpha,\lambda})_{\alpha\in N_s,\lambda\in \Lambda_\mathbb C,}$ with 
		$$
		\sum\limits_{\lambda\in\Lambda_{\mathbb C},\alpha\in N_s}|a_{k,\alpha,\lambda}|^2e^{-\pi k |\lambda|^2}=1
		$$
		which we will interpolate. We define the function 
		
		\begin{align*}
			v(z):=\sum\limits_{\lambda\in\Lambda_{\mathbb C},\alpha\in N_s}a_{k,\alpha,\lambda} e^{k\pi\overline{\lambda}(z-\lambda)}\frac{(z-\lambda)^{\alpha}}{\alpha!}\chi(|z-\lambda|),
		\end{align*}
		where $\chi$ is a cutoff function like before. Now it is simply a matter of repeating the process from the uniqueness description of the pseudoeffective threshold (or \cite[Theorem 2.8]{LW}) to obtain a global function in the Bargmann Fock space whose jets will interpolate the sequence $a_{k,\alpha,\lambda}.$
		Thus we get that
		$$
		\sigma_k(\mathbb C^n/\Lambda_\mathbb C,\omega)>\lfloor k\gamma\rfloor-n 
		$$
		Hence we infer successively, 
		$$
		\frac{\sigma_k(\mathbb C^n/\Lambda_\mathbb C,\omega)}{k}>{\epsilon_0(\mathbb C^n/\Lambda_\mathbb C,\omega})-\frac{n+1}{k},
		$$
		and
		$$
		\liminf\limits_{k\to \infty} \frac{\sigma_k(\mathbb C^n/\Lambda_\mathbb C,\omega)}{k}\ge{\epsilon_0(\mathbb C^n/\Lambda_\mathbb C,\omega}).
		$$
        \end{proof}

\section{STFT with Hermitian window}
\noindent In order to connect our results on the Seshadri constant and pseudoeffective threshold to Gabor frames we will use the Bargmann transform. We define the Bargmann transform here by
\begin{align*}
    \mathcal{B}f(z):=2^{n/4}\int_{\mathbb R^n}f(t)e^{-\pi t^Tt-2\pi z^Tt-\pi \frac{z^Tz}{2}}.    
\end{align*}
The key to connecting this to Hermite windows is to define a "time-frequency shift" on the Bargmann-Fock space. This will be the intertwining operator defined by
\begin{align*}
    \pi_z^{\mathbb C}F(\zeta):=\exp(i\pi \xi x-\frac{\pi}{2}|z|^2)e^{\zeta z}F(\zeta-\overline{z}),\quad z=x+i\xi
\end{align*}
which fulfills the identity
\begin{align*}
    \pi^\mathbb C_z \mathcal{B}f = \mathcal{B}\pi_{(\xi,x)}f, \quad f\in L^2(\mathbb R^n).
\end{align*}
The Hermite functions allow us to generalize the correspondence between the STFT with Gaussian window and the Bargmann Transform. Following \cite[Proposition 1]{Abreu10} (which is a higher dimensional generalization of \cite[Proposition 3.2]{GL09})
we have that for any $f\in L^2(\mathbb R^n)$,
\begin{align}\label{polydualcomp}
V_{h_\alpha}f(-\xi,x):=&(f, \pi_{(-\xi,x)} h_\alpha )\\
=&(Bf,\pi^\mathbb C_{\overline{z}} Bh_\alpha )_{\mathcal{F}^2}\nonumber \\
= &(\pi^{|\alpha|}\alpha!)^{-1/2}\exp\left(i\pi x^T\xi-\frac{\pi}{2}|z|^2\right)e^{\pi|z|^2} \frac{d^\alpha}{dz^\alpha}\left( e^{-\pi|z|^2} \mathcal B f(z)\right), \quad \alpha\in \mathbb N_0^n \nonumber
\end{align}
where $z=x+i\xi$ and $\mathcal{B}f$ is the standard Bargmann transform of $f$. If we define the true polyanalytic Bargmann transform of $f$ as
$$
\mathcal{B}^\alpha f(z):= (\pi^{|\alpha|}\alpha!)^{-1/2} e^{\pi|z|^2} \frac{d^\alpha}{dz^\alpha}\left( e^{-\pi|z|^2} \mathcal B f(z)\right), \quad \alpha\in \mathbb N_0^n,
$$
we have
$$
V_{h_\alpha}f(-\xi,x) = \exp\left(i\pi x^T\xi-\frac{\pi}{2}|z|^2\right)\mathcal{B}^\alpha f(z).
$$
\noindent In particular, it follows from this definition that the polyanalytic Bargmann transform is an isometry between $L^2(\mathbb R^d)$ and $L^2(\mathbb C^d, e^{-\pi|z|^2})$. We remark that this setup requires that we use the matrix $iI\in \mathfrak h$ to define our complex structure. How to generalize this to a general element from the Siegel upper half space is unclear at this time.
\\~~\\
Now by looking at the relationship between the STFT for hermitian window and the true polyanalytic Bargmann transform we can easily see the following.
\begin{lemma}[{\cite[Lemma 3]{Abreu10}}]\label{le:Sampling_Multiwindow}
    The system $(\mathbf{h}_s,\Lambda)$ induces a multiwindow Gabor frame if and only if there exist constants $A,B>0$ such that for any $F\in \mathcal{F}^2$ we have that 
    \begin{align*}
        A\|F\|_{\mathcal{F}^2}\le\sum\limits_{\lambda\in\Lambda_\mathbb C}\sum\limits_{|\alpha|<s}|(F,\mathcal\pi^\mathbb C_{\overline{\lambda}}\mathcal{B}h_\alpha)_{\mathcal{F}^2}|^2\le B\|F\|_{\mathcal{F}^2}.
    \end{align*}
\end{lemma}
We remark that by \cite[Section 5]{Abreu10} this condition is equivalent to interpolation of the "full" polyanalytic Bargmann transform of level $s$ given by
\begin{align*}
    \mathbf{B}^s((f_\alpha)_{|\alpha|<s}):=\sum\limits_{\alpha}\mathcal{B}^\alpha f_\alpha.
\end{align*}
In order to relate Gabor frames to our sampling numbers we will use the characterization of Gabor frames without inequalities. By \cite{G07} we have the following characterization (note that the Hermite functions are Schwartz class).
\begin{lemma}
    The system $(\mathbf{h}_s,\Lambda)$ induces a multiwindow Gabor frame if and only if the coefficient operator defined by
    \begin{align*}
        \mathcal{C}_{\mathbf{h}_s,\Lambda}: f \mapsto \{(f,\pi_{\lambda}h_\alpha)_{|\alpha|\le s}\}_{\lambda\in\Lambda}
    \end{align*}
    is one to one between the spaces $M^\infty$ of tempered distributions with $\|V_{h_0}f\|_\infty<\infty$ and the space of bounded sequences. In other words $\Lambda$ induces a multiwindow Gabor frame if and only if for any such tempered distribution $f$ we have that $(f,\pi_\lambda h_\alpha)=0$ for all $\lambda\in \Lambda,\alpha\in N_s,$ implies that $f=0$.
\end{lemma}
\noindent By the same rationale as Lemma \ref{le:Sampling_Multiwindow}, this implies the following characterization.
\begin{corollary}\label{cor:MWUnique}
    The system $(\mathbf{h}_s,\Lambda)$ induces a multiwindow Gabor frame if and only if we have that for $F\in \mathcal{F}^\infty$, the condition 
    $$
    (F,\pi^\mathbb C_{\overline{\lambda}}\mathcal{B}h_\alpha) = 0, \quad \lambda\in \Lambda_{\mathbb C}, \alpha\in N_s,
    $$
    implies that $F\equiv 0$.
\end{corollary}
In order to find criteria for super frames we will proceed similarly, however, we will notably need the Ron-Shen duality princple (see \cite{RS, JL}). We state the version we will use here:
\begin{theorem}
    The system $(\mathbf{h}_s,\Lambda)$ induces a super Gabor frame if and only if the system $(\mathbf{h}_s,\Lambda^\circ)$ induces a multiwindow Riesz sequence, i.e. any sequence $c\in \ell^2(\Lambda^\circ,\mathbb C^\mathfrak s)$ we have constants $A,B>0$ such that
    \begin{align}
        A\sum\limits_{\lambda\in\Lambda^\circ}\sum_{\alpha\in N_s}|c_{\lambda,\alpha}|^2\le \left|\sum\limits_{\lambda\in\Lambda^\circ}\sum_{\alpha\in N_s} c_{\lambda,\alpha}\pi_\lambda h_\alpha\right|^2\le B\sum\limits_{\lambda\in\Lambda^\circ}\sum_{\alpha\in N_s}|c_{\lambda,\alpha}|^2.
    \end{align}
    where we define $\Lambda^o$ as the symplectic dual of $\Lambda$:
    \begin{align*}
        \Lambda^\circ:=\{(\eta,y)\in\mathbb R^{2n} : \eta x-\xi y\in \mathbb Z, (\xi,x)\in \Lambda\}.
    \end{align*}
\end{theorem}
Similarly to the uniqueness condition for Gabor frames, the condition of a Riesz sequence can be connected to an interpolation condition. We write the following (compare with \cite{Abreu10,G07}).
\begin{lemma}\label{le:SFInterp}
    The system $(\mathbf{h}_s,\Lambda)$ induces a super Gabor frame if and only if for any sequence $c_{\lambda,\alpha}\in \ell ^2(\Lambda_\mathbb C ^\circ,\mathbb C^{\mathfrak s})$ we have that there exists a function $F\in\mathcal{F}^2$ such that
    \begin{align}
        (F,\pi_{\overline{\lambda}}\mathcal{B}h_\alpha) = a_{\lambda,\alpha}, \quad\lambda\in\Lambda_\mathbb C^o, \alpha\in N_s.
    \end{align}
    We remark that this is equivalent to the sampling property for the full Bargmann transform by \cite{Abreu10}.
\end{lemma}
\section{Connection to Seshadri and Pseudoeffective Thresholds}
It is not difficult to see from formula (\ref{polydualcomp}) that the jet interpolation given by the Seshadri constant implies interpolation of the terms $(F,\pi_{\overline{\lambda}}\mathcal{B}h_\alpha)$. Hence we have the following as a consequence of Lemma \ref{le:SFInterp}.
\begin{proposition}
    Consider the dual lattice to $\Lambda_\mathbb C$ given by
    $$
    \Lambda^\circ_{\mathbb C}:=\{\eta+iy \in \mathbb C^n: \xi^T y -x^T\eta \in \mathbb Z, \ \forall\  (\xi,x) \in \Lambda\}.
    $$
    For any $s\in \mathbb N_0$, 
    $$
    \sigma_1(\Lambda_\mathbb C^\circ)\ge s
    $$
    implies that the system 
    $(\mathbf{h}_s:=(h_\alpha)_{|\alpha|\le s},\Lambda)$ induces a super-Gabor frame. In particular, we have that the superframe property is fulfilled whenever $\epsilon(\mathbb C^n/\Lambda_\mathbb C^o,\pi |z|^2)>n+s$.
\end{proposition}

\noindent Similarly if a function $F\in\mathcal{F^\infty}$ fulfills 
$$
(F,\pi^\mathbb C_{\overline{\lambda}}h_\alpha)=0
$$
for each individual $\alpha$, we may see by linear combination of these conditions that we in fact have $\tfrac{\partial^\alpha F}{\partial z^\alpha}|_\lambda=0$. Now we apply Corollary \ref{cor:MWUnique} and lo, we have the following.  

\begin{proposition}
For any $s\in \mathbb N_0$, 
$$
\mu_1(\Lambda_\mathbb C)\le s  
$$
implies that $\Lambda$ induces a multiwindow-Gabor frame with respect to the system  $(\mathbf{h}_s:=(h_\alpha)_{|\alpha|\le s},\Lambda)$. In particular, we have a multiwindow frame when $\lambda(\mathbb C^n/\Lambda_\mathbb C,\pi |z|^2)<s+1$.
\end{proposition}

\section{The Case of Transcendental Lattices}
Now that we have frame criteria in terms of the Seshadri constant and pseudoeffective threshold, the objective is to bound these from below and above,  respectively. This however, is not so easy since the magnitudes of these numbers are dependent on the nature of the curves on the manifold for which they are defined. However, if we make an assumption which "removes the geometry" of the tori we examine, there is quite a bit more we can say. This is precisely the condition of a transcendental lattice which we examine further in this section.
\begin{proposition}\label{pr:GL1} If $\mathbb C^n/\Lambda_{\mathbb C}$ has no analytic subvariety of dimension $1\leq d< n$ then $\epsilon_0(\mathbb C^n/\Lambda_\mathbb C,\omega) = \lambda_0(\mathbb C^n/\Lambda_\mathbb C,\omega) = (n!|\Lambda|)^{1/n}$. 
\end{proposition}

\begin{proof} 

Consider an $\omega$-PSH function $G$ with nonzero Lelong number at the lattice point. We clearly have that $dd^cG(z)$ is bounded from below by $-\omega$, meaning we may use the approximation result of Demailly for quasi-psh functions (\cite[Proposition 3.7]{Dem92}) to find a sequence of quasi-PSH functions $(G_m)_m$ which converge to $G$, such that $G_m>G$, have only analytic singularities, and fulfill the estimate:

$$
\nu_x(G)-\frac{n}{m}\leq \nu_x(G_m) \leq \nu_x(G),\; x \in \mathbb C^n/\Lambda_{\mathbb C}
$$
for the Lelong numbers of $G$ and $G_m$. We have the estimate for $G_m$
$$
dd^cG_m \geq -(1+\varepsilon_m)\omega,
$$
for some sequence of numbers $(\varepsilon_m)_m$ which decreases to $0$.
\\~~\\
 Assume now that $\mathbb C^n/\Lambda_{\mathbb C}$ has no analytic subvariety of dimension $1\leq d< n$. This then implies that the singularities of the $G_m$ must be isolated. Thus we can bound the Seshadri constant of $\mathbb C^n/\Lambda_{\mathbb C}$ from below by the Lelong numbers of $\frac{G_m}{1+\varepsilon_m}$  at the point corresponding to the origin.  Since these are bounded from below by $\frac{\nu_0(G)-{(n/m)}}{1+\varepsilon_m}$, taking the limit as $m$ approaches infinity, we have that the Seshadri constant of $\mathbb C^n/\Lambda_{\mathbb C}$ is no less than the Lelong number of $G$. Since the only restriction on the Lelong number of $G$ is that it is less than the pseudoeffective threshold, we get that the Seshadri constant is no smaller than the pseudoeffective threshold. 
 \\~~\\
 Demailly's mass concentration trick also implies that (see \cite[Theorem 2.8]{Tosatti1}) $\epsilon_0(\mathbb C^n/\Lambda_\mathbb C)=(n!|\Lambda|)^{1/n}$.
\end{proof}

To check that $\mathbb C^n/\Lambda_{\mathbb C}$ has no analytic subvariety of dimension $1\leq d< n$ we need the following result (the proof follows directly from the argument in \cite[pp.164-165]{Sha}).

\begin{lemma}\label{le:sha} Let $\Gamma$ be a lattice in $\mathbb C^n$. Assume that the following linear mapping from singular homology to functionals on Dolbeault cohomology
\begin{equation}\label{eq:sha}
{\rm int}_\Gamma: \oplus_{1\leq k< n} H_{2k} (\mathbb C^n/\Gamma, \mathbb Z) \to \oplus_{p+q=2k, \,1\leq k< n, \,p>q\geq 0} (H^{p,q} (\mathbb C^n/\Gamma, \mathbb C))^*
\end{equation}
defined by
$$
{\rm int}_\Gamma(S)(\alpha):=\int_S \alpha
$$
is injective. Then $\mathbb C^n/\Gamma$ has no analytic subvariety of dimension $1\leq d< n$.
\end{lemma}

Let us give a more elementary description of this condition. The lattice $\Gamma$ is given by
$$
\Gamma = \mathbb{Z}e_1 + ... + \mathbb{Z}e_{2n},
$$
where $e_1,...,e_{2n}\in\mathbb{C}^n$ are vectors which are linearly independent over $\mathbb{R}$. In this case, the homology groups $H_{2k} (\mathbb C^n/\Gamma, \mathbb Z)$ have bases over $\mathbb Z$ given by
$$
S_i := \pi(\mathbb R e_{i_1} + ... + \mathbb R e_{i_{2k}}),
$$
for the multi-index $i=(i_1,...,i_{2k}) \in \mathbb N ^{2k}$ with $i_j<i_l$ for $j<l$. (For the sake of convenience we will denote the set of such indices by $I_k$). Observe that the map $\pi$ is a projection from $\mathbb C^n$ onto the complex torus $X$. As such, we may now write any cycle $S \in H_{2k} (\mathbb C^n/\Gamma, \mathbb Z)$ as 
$$
S=\sum\limits_{i\in I_k} a_iS_i,
$$
for some integers $a_i$. This fact should also help explain the singular homology groups on the Torus for those who are unfamiliar: $H_{2k}(\mathbb C^n/\Gamma,\mathbb Z)$ can be thought of as the free Abelian group over the closed "curves" of real dimension $2k$ (modulo some equivalence relation which will not be relevant in what follows). Now, for a $(p,q)$-form $\alpha$, where $p+q=k$ and $p>q\geq0$, we have that 
$$
 {\rm int}_\Gamma(S)(\alpha) = \sum\limits_{i\in I_k}a_i\int_{S_i}\alpha.
$$
Hence the injectivity of ${\rm int}_\Gamma(S)$ is equivalent to the following statement: 
\\~\\
If for every closed $(p,q)$ form $\alpha$, where $p+q=k$, $p>q\geq0$ ${\rm int}_\Gamma(S)(\alpha)=0$, the numbers $a_i$ must all be 0. The reader should of course note that this condition must also be checked for every $k$ less than $n$. Furthermore,  by linearity  it  suffices to check this condition on a a basis of $H^{p,q}(\mathbb C^n / \Gamma,\mathbb C)$.
\\~\\
We now turn to the case where $n=2$. Hence we need only to check the case $k=1$, which also forces $p=2$ and $q=0$. Additionally, the $(2,0)$ forms are spanned by the single form induced by $dz_1\wedge dz_2$. As such, the condition that ${\rm int}_\Gamma(S)(\alpha)$ is injective can be expressed as the condition that the numbers
$$
\int_{S_i}\pi_*(dz_1\wedge dz_2)
$$
are linearly independent over $\mathbb Z$. Explicitly computing these integrals gives the following result:

\medskip

\noindent
\textbf{Remark.}  \emph{In case $n=2$ and $\Gamma$ generated by
$e_j=(\alpha_j, \beta_j)$, $1\leq j\leq 4$, over $\mathbb Z$, then \cite[p.165]{Sha} tells us that $\ker {\rm int}_\Gamma=0$ if and only if 
$$
\alpha_j\beta_k-\alpha_k\beta_j, \ \\ 1\leq j<k\leq 4,
$$
are linearly independent over $\mathbb Z$. }

\medskip

For $n>2$, we have way more conditions. When we examine the case $k>1$, we must not only look at the group $H^{2k,0} (\mathbb C^2/\Gamma, \mathbb C)$, but also  the groups $H^{2k-1,1} (\mathbb C^2/\Gamma, \mathbb C)$ and so on. The group $H^{p,q} (\mathbb C^2/\Gamma, \mathbb C)$ has a basis given by the forms
$$
\pi_*(dz_P\wedge d\overline{z}_Q), \; P\in\mathbb N^p,Q\in\mathbb N^q,
$$
when we take $P$ and $Q$ to be ascending. We can compute the integrals
$$
\int_{S_i}\pi_*(dz_P\wedge d\overline{z}_Q)
$$
in a manner similar to the one which gives us the statement in the previous remark, and see that these are the numbers
$$
C^i_{P,Q}:=\mathrm{det}
\left(
\begin{matrix}
e_{i_1,P_1}&...&e_{i_1,P_p} &\overline{e}_{i_1,Q_1}&...&\overline{e}_{i_1,Q_q}\\
.& & & & &.\\
.& & & & &.\\
.& & & & &.\\
e_{i_{2k},P_1}&...&e_{i_{2k},P_p} &\overline{e}_{i_{2k},Q_1}&...&\overline{e}_{i_{2k},Q_q}
\end{matrix}
\right).
$$
We may now reformulate our criterion for the nonexistence of subvarieties of the complex torus in the following way.

\begin{lemma}
We have $\rm{ker}$ $\rm int_{\Gamma} =0$ if and only if for every $0<k<n$ there exists no set of integers $\{a_i\}_{i\in I_k} \subseteq \mathbb Z$ not all equal to $0$ such that for all ascending multi-indices $P,Q$ where $(P,Q)\in \mathbb N^{2k}$ and $P$ has more entries than $Q$,
$$
\sum\limits_{i\in I_k}a_iC^i_{P,Q}=0.
$$
\end{lemma}
 
This immediately gives the following sufficient criterion which may be easier to check than the original condition. 

\begin{corollary}
$\rm{ker} \,\rm int_{\Gamma} =0$ if for every $0<k<n$ there exist multi-indices $P,Q$ where $(P,Q)\in \mathbb N^{2k}$ and $P$ has more entries than $Q$ are such that the numbers $\{C^i_{P,Q}\}_{i\in I_k}$ are linearly independent over $\mathbb Z$.
\end{corollary}

In general,  Lemma \ref{le:sha} gives the following explicit version of \cite[Thm 1.1]{LW}.

\begin{theorem}[{Theorem \ref{th:mainthm}}]\label{th:extra} We have that: 
\begin{enumerate}
\item[(1)] $(\mathbf h_s, \Lambda)$ is a multiwindow-Gabor frame for $L^2(\mathbb R^n)$ if $|\Lambda| <\frac {s^n}{n!}$ and $\ker {\rm int}_{\Lambda_{\mathbb C}} =0$.

\item[(2)] $(\mathbf h_s, \Lambda)$ is a super Gabor frame for $L^2(\mathbb R^n, \mathbb C^{\mathfrak s})$ if $|\Lambda| <\frac{n!}{(n+s)^n}$ and $\ker {\rm int}_{\Lambda^\circ_{\mathbb C}} =0$.
\end{enumerate}
In particular, when $s=0$, both (1) and (2) will give a criterion for $(e^{-\pi |t|^2},\Lambda)$ to be a Gabor frame.
\end{theorem}

\noindent
\textbf{Remark.}  \emph{Almost all lattices in $\mathbb R^{2n}$ satisfy $\ker {\rm int}_{\Lambda_{\mathbb C}} =0$ and $\ker {\rm int}_{\Lambda^\circ_{\mathbb C}} =0$ (as we will show later on), hence the above theorem gives an effective Gabor frame criterion in terms of the covolume for almost all lattices. }

\noindent
\begin{corollary} \label{cor:RUZ}
In case $\Lambda = \mathbb Z^2 \times A \mathbb Z^2$, where $A$ is a real linear mapping defined by
$$
A(x,y):=(ax+by, cx+dy),
$$
then we have that 
$$
\ker {\rm int}_{\Lambda_{\mathbb C}} =0 \Leftrightarrow\ker {\rm int}_{\Lambda^\circ_{\mathbb C}} =0 \Leftrightarrow ad-bc \notin \mathbb Q, \ a,b,c,d \, \text{are $\mathbb Z$-linearly independent.} 
$$
Such a lattice then induces a Gabor frame if additionally, $|ad-bc|<\frac{1}{2}$. 
\end{corollary}
We remark that the examples in Theorem 1.5 in \cite{RUZ} do not satisfy our assumptions: all $a,b,c,d$ are dependent over $\mathbb Z$, and $|ad-bc|$ need not be less than ${1}/{2}$. However, as shown in \cite{RUZ}, these lattices still form Gabor frames.

\begin{proof}
    In this case the complexified  lattice $\Lambda_{\mathbb C}$ is generated by the vectors
    \begin{align*}
        \begin{pmatrix}
            1\\
            0
        \end{pmatrix},
        \begin{pmatrix}
            0\\
            1
        \end{pmatrix},
        \begin{pmatrix}
            ia\\
            ic
        \end{pmatrix},
        \begin{pmatrix}
            ib\\
            id
        \end{pmatrix}.
    \end{align*}
    We may now apply the two dimensional version of our criterion for a transcendental lattice which was given in an earlier remark to see that $\Lambda$ is transcendental precisely when the determinants $1$, $bc-ad$, $ia$, $ib$, $ic$, and $id$ are linearly independent over $\mathbb Z$. This happens precisely when the condition given in the statement of the Lemma is fulfilled. The case of the symplectic dual can be verified in the same way.
\end{proof}

\medskip

We will now prove the following genericity result:
\begin{theorem}
Transcendental tori are generic with respect to the choice of the lattice.
\end{theorem} 
We will first clarify what we mean by generic.  Consider the fact that the space of lattices is equal to the space $\mathrm{Gl}(2n,\mathbb R)/\mathrm{Gl}(2n,\mathbb Z)$. We say that a property is generic for lattices if there exists an $S\subsetneq \mathrm{Gl}(2n,\mathbb R)$,  which is a set of Lebesgue measure zero in $\mathbb R^{4n^2}$ such that the property is fulfilled by all lattices in $(\mathrm{Gl}(2n,\mathbb R)\setminus S)/\mathrm{Gl}(2n,\mathbb Z)$.
\begin{proof}
Consider any set of integers $\{a_i\}_{i\in I_k}$. Then the equations
$$
\sum\limits_{i\in I_k}a_iC^i_{P,Q}=0
$$
induce a closed variety in $\mathbb R^{4n^2}.$ It is sufficient to prove that a variety of this form has codimension of at least 1 for some given $(P,Q)$, as this will imply that the variety in this case is of Lebesgue measure zero , and the set in $\mathbb R^{4n^2}$ which does not induce a transcendental torus must be contained within this union. This is a simple observation to make since all we must show is that there is a point which does not lie in this subvariety. 

\smallskip

To see this,  observe in the case $2k\leq n$,  and choose $P$ to be of length $2k$ and $Q$ an empty multiindex. Then the polynomials $C^i_{P,Q}$ are linearly independent over $\mathbb R[e_{1,1},...,e_{2n,n}]$. 

\smallskip

We see this by looking at the "first term" of the $C^i_{P,Q}.$ By this, we mean the summand in the determinant which is given by multiplying the elements on the main diagonal.  So this term will be of the form
$$
e_{i_1,P_1}\cdot...\cdot e_{i_{2k},P_{2k}}.
$$
If we now choose a different index $j\in I_k$. Without loss of generality we may assume that the number $i_1$ does not occur in $j$.  This means that the variable $e_{i_1P_1}$ also cannot occur in $C^j_{P,Q}$, making $C^j_{P,Q}$ linearly independent from $C^i_{P,Q}$. Thus, nontrivial linear combinations of the $C^i_{P,Q}$ cannot vanish identically. 

\smallskip

Now observe the case where $2k>n$. We now take $P=(1,2,...n)$, and $Q$ to be of length $2k-n$. We need to check that the corresponding polynomials $C^i_{P,Q}$ are linearly independent over $\mathbb R[e_{1,1}, \bar e_{1,1}, ..,e_{2n,n}, \bar e_{2n, n}]$. 

\smallskip

Once again, this follows from that the term given by multiplying the elements of the diagonal, $e_{i_1,1}\cdot...\cdot e_{i_n,n}\cdot \overline{e}_{i_{n+1},Q_1}\cdot...\cdot \overline{e}_{i_{2k},Q_{2k-n}}
$
is unique with respect to $i$. 

\end{proof}

\end{document}